\documentclass{article}

\usepackage{amsmath,amssymb,amsfonts,amsthm}

\newtheorem{theorem}{Theorem}[section]
\newtheorem{hypothesis}[theorem]{Hypothesis}
\newtheorem{proposition}[theorem]{Proposition}
\newtheorem{lemma}[theorem]{Lemma}
\numberwithin{equation}{section}

\theoremstyle{remark}
\newtheorem{remark}[theorem]{Remark}

\newcommand{\Ric}{\mathop{\mathrm{Ric}}\nolimits}
\newcommand{\II}{\mathop{\mathrm{II}}\nolimits}
\newcommand{\Ad}{\mathop{\mathrm{Ad}}\nolimits}
\newcommand{\pr}{\mathrm{pr}}
\newcommand{\Hess}{\mathop{\mathrm{Hess}}\nolimits}
\newcommand{\Div}{\mathop{\mathrm{div}}\nolimits}

\author{Artem Pulemotov\thanks{School of Mathematics and Physics, The
University of Queensland, St Lucia,~QLD 4072, Australia} \\
\small{\texttt{a.pulemotov@uq.edu.au}}}

\title{The Ricci flow on domains in cohomogeneity one manifolds}

\begin{document}

\maketitle

\begin{abstract}
Suppose $G$ is a compact Lie group, $H$ is a closed subgroup of $G$, and the homogeneous space $G/H$ is connected. The paper investigates the Ricci flow on a manifold $M$ diffeomorphic to $[0,1]\times G/H$. First, we prove a short-time existence and uniqueness theorem for a $G$-invariant solution $g(t)$ satisfying the boundary condition $\II(g(t))=F(t,g_{\partial M}(t))$ and the initial condition $g(0)=\hat g$. Here, $\II(g(t))$ is the second fundamental form of $\partial M$, $g_{\partial M}$ is the metric induced on $\partial M$ by $g(t)$, $F$ is a smooth map and $\hat g$ is a metric on $M$. Second, we study Perelman's $\mathcal F$-functional on~$M$. Our results show, roughly speaking, that $\mathcal F$ is non-decreasing on a $G$-invariant solution to the modified Ricci flow, provided that this solution satisfies boundary conditions inspired by the~2012 paper of Gianniotis.

\end{abstract}

\section{Introduction}

Developing the theory of the Ricci flow on manifolds with boundary is a long-standing open problem with numerous potential applications. The present paper addresses several aspects of this problem in the setting of spaces with symmetries. We focus on the short-time existence and the uniqueness of solutions, as well as the monotonicity of Perelman's $\mathcal F$-functional. Before we describe our results, let us review the history of the subject.

Suppose $M$ is a $d$-dimensional manifold with $d\ge3$. The Ricci flow on $M$ is the partial differential equation
\begin{align}\label{intro_RF_eq}
\frac{\partial}{\partial t}g(t) = -2\Ric(g(t))
\end{align}
for a Riemannian metric $g(t)$ depending on the parameter $t\geq 0$. It is customary to interpret $t$ as time. In the right-hand side, $\Ric(g(t))$ stands for the Ricci curvature of $g(t)$. Given a Riemannian metric $\hat g$ on $M$, we supplement~\eqref{intro_RF_eq} with the initial condition
\begin{align}\label{intro_iniC}
g(0)=\hat g.
\end{align}
The Ricci flow on manifolds without boundary
has been widely studied. The reader will find a wealth of information about it
in the books~\cite{BCPLLN06,PT06,JMGT07}.

Assume the manifold $M$ is compact and $\partial M$ is nonempty. When trying to develop the theory of the Ricci flow on $M$ in this case, one faces a number of major roadblocks. For instance, it is necessary to find boundary conditions for equation~\eqref{intro_RF_eq} that would allow tractable analysis and admit a meaningful geometric interpretation. Doing so is difficult, as equation~\eqref{intro_RF_eq} is only weakly parabolic; see~\cite[Section~5.1]{PT06}, \cite[Introduction]{PG12} and also~\cite[Secion~3]{MA08}.
Note that the Ricci flow on \emph{surfaces} with boundary appears to be more approachable. For results in this area, consult the references in~\cite{PG12}. However, the higher-dimensional setting considered in the present paper encompasses a different set of difficulties and requires different techniques.

Initial progress regarding the Ricci flow on $M$ under the above assumptions was made by Shen in~\cite{YS92,YS96}. 
Let us briefly describe it. Suppose $\II(g(t))$ is the second fundamental form of $\partial M$ with respect to $g(t)$, $g_{\partial M}(t)$ is the metric induced on $\partial M$ by $g(t)$, and $\lambda$ is a fixed real number. Shen considered equation~\eqref{intro_RF_eq} subject to the boundary condition
\begin{align}\label{intro_ShenBC}
\II(g(t))=\lambda g_{\partial M}(t)
\end{align}
and the initial condition~\eqref{intro_iniC}. He outlined the proof of a short-time existence theorem. He also obtained a formula for the $t$-derivative of $\II(g(t))$. Despite this success, further investigation of problem~\eqref{intro_RF_eq}--\eqref{intro_ShenBC}--\eqref{intro_iniC} turned out to be complicated. As of today, no uniqueness theorem for this problem is found in the literature. The long-time behaviour of solutions was investigated in~\cite{YS92,YS96,XCTD06,JC09}, but it is not yet deeply understood. In fact, the majority of available results concern the special case where $\lambda=0$ (i.e., the boundary is totally geodesic). We invite the reader to see~\cite{MBXCAP10} for a discussion about letting the parameter $\lambda$ depend on $t$. Aside from the discussion, that paper contains two gradient estimates for the heat equation under~\eqref{intro_RF_eq}--\eqref{intro_ShenBC}.

Further progress in the study of the Ricci flow on $M$ was made by the author in~\cite{AP10}. More precisely, suppose $\mathcal H(g(t))$ is the mean curvature of $\partial M$ with respect to~$g(t)$. Fix a function $b:[0,\infty)\to\mathbb R$. The main result of~\cite{AP10} is a short-time existence theorem for solutions to~\eqref{intro_RF_eq} under the boundary condition 
\begin{align}\label{intro_AP_BC}
\mathcal H(g(t))= b(t)
\end{align}
and the initial condition~\eqref{intro_iniC}. This result was improved by Gianniotis in the paper~\cite{PG12}. More specifically, choose a $t$-dependent Riemannian metric $\eta(t)$ and a $t$-dependent real-valued function $\kappa(t)$ on $\partial M$. Gianniotis considered the Ricci flow subject to the boundary conditions
\begin{align}\label{intro_Gianniotis_BC}
[g_{\partial M}(t)] = [\eta(t)],\qquad \mathcal H(g(t)) = \kappa(t).
\end{align}
The square brackets here denote the conformal class. We emphasize that, in contrast with~\eqref{intro_AP_BC}, formulas~\eqref{intro_Gianniotis_BC} allow the mean curvature $\mathcal H(g(t))$ to be nonconstant on $\partial M$. The reasoning in~\cite{PG12} yielded the short-time existence and the uniqueness of solutions to~\eqref{intro_RF_eq}--\eqref{intro_Gianniotis_BC}--\eqref{intro_iniC} under natural assumptions. At the same time, describing the behaviour of these solutions for large~$t$ remains a challenging open problem. In the recent work~\cite{PG13}, Gianniotis made progress towards the resolution of this problem by producing several interesting estimates. However, a comprehensive long-time existence theorem is still out of reach. Note that Gianniotis's results were largely inspired by Anderson's work on the Einstein equation; see~\cite{MA08}.

Another direction in the study of the Ricci flow on manifolds with boundary is the analysis of Perelman's $\mathcal F$-functional. The reader may consult, e.g.,~\cite[Chapter~6]{PT06} for the definition and the key properties of~$\mathcal F$. Given a $t$-dependent Riemannian metric $g(t)$ and a $t$-dependent real-valued function $p(t)$ on $M$, it is well-known that the expression 
\begin{align*}
\mathcal F(g(t),p(t)) = \int_M(R(g(t)) + |\nabla p(t)|^2)e^{-p(t)}\,d\mu
\end{align*}
would be non-decreasing in $t$ if $\partial M$ were empty and the pair $(g(t),p(t))$ satisfied
\begin{align}\label{intro_MRF}
\frac{\partial}{\partial t}g(t) &= -2(\Ric(g(t)) + \Hess p(t)),\notag\\
\frac{\partial}{\partial t}p(t) &= -\Delta p(t) - R(g(t)).
\end{align}
(In the equalities above, $R$ and $\mu$ denote the scalar curvature and the volume measure.) In fact, one would be able to interpret system~\eqref{intro_MRF} as the gradient flow of $\mathcal F$. 
The metric $g(t)$ would be the pullback of a solution to~\eqref{intro_RF_eq} by a $t$-dependent diffeomorphism. Monotonicity properties of $\mathcal F$ are substantially harder to discover when $\partial M\ne\emptyset$. Lott's paper~\cite{JL13} provides several formulas for $\frac d{dt}\mathcal F(g(t),p(t))$ assuming the pair $(g(t),p(t))$ satisfies~\eqref{intro_MRF} and, after appropriate diffeomorphisms are performed, $\partial M$ evolves under the mean curvature flow. The works~\cite{JC07,JC09,JCAM12} contain related computations. However, none of the results in~\cite{JC07,JC09,JCAM12,JL13} asserts that $\mathcal F(g(t),p(t))$ is non-decreasing.

The present paper focuses on boundary-value problems for the Ricci flow~\eqref{intro_RF_eq} under the assumption
\begin{align}\label{intro_formM}
M\simeq[0,1]\times G/H,
\end{align}
where $G$ is a compact Lie group, $H$ is a closed subgroup of $G$, and $G/H$ is connected. In a sense, equality~\eqref{intro_formM} means $M$ possesses axial symmetry. The boundary of $M$ has two connected components. Spaces of the form~\eqref{intro_formM} arise as (closures of) domains on cohomogeneity one manifolds. Recently, the author used them in the study of the prescribed Ricci curvature equation; see~\cite{AP13} and also~\cite{AP11}. It is worth mentioning that cohomogeneity one manifolds enjoy numerous applications in geometry and mathematical physics. In particular, they have been used to construct important examples of Einstein metrics; see, e.g.,~\cite{ADMW99} and references therein. They were effectively employed in the paper~\cite{ADMW11} to investigate Ricci solitons. For more information on the basic properties and applications of cohomogeneity one manifolds, consult~\cite{CH10}.

The literature devoted to the Ricci flow on spaces with symmetries is rather extensive. The papers~\cite{JoLa13,MB14} are two examples of recent works on the subject. The introduction to~\cite{JLNS14} contains a survey of what is known in three dimensions. The vast majority of existing works, however, only consider manifolds without boundary.

Let us describe our results. In what follows, we assume $M$ has the form~\eqref{intro_formM} and the isotropy representation of $G/H$ splits into pairwise inequivalent irreducible summands. The latter assumption is quite standard in the theory of cohomogeneity one manifolds; we will discuss it in detail before stating our first theorem. Section~\ref{sec_STE_uniq} considers the Ricci flow~\eqref{intro_RF_eq} on $M$ subject to the boundary condition
\begin{align}\label{intro_my_BC}
\II(g(t))=F(t,g_{\partial M}(t)).
\end{align}
The letter $F$ here denotes a map with values in the space of symmetric $G$-invariant (0,2)-tensor fields on~$\partial M$. When $F(t,g_{\partial M}(t))=\lambda g_{\partial M}(t)$, formula~\eqref{intro_my_BC} becomes Shen's boundary condition~\eqref{intro_ShenBC}. In Section~\ref{sec_STE_uniq}, we establish the short-time existence and the uniqueness of $G$-invariant solutions to problem~\eqref{intro_RF_eq}--\eqref{intro_my_BC}--\eqref{intro_iniC} assuming~\eqref{intro_my_BC} holds at $t=0$ and $\hat g$ is $G$-invariant. The author intends to study the behaviour of these solutions for large $t$ in subsequent papers. Note that, until now,~\eqref{intro_Gianniotis_BC} has been the only boundary condition known to guarantee both the short-time existence and the uniqueness for the Ricci flow with given initial data in dimensions three or higher.

Section~\ref{sec_MotFF} studies the Perelman $\mathcal F$-functional on $M$. We begin with an examination of system~\eqref{intro_MRF} subject to the boundary conditions~\eqref{intro_Gianniotis_BC} on $g(t)$ and the Neumann condition $\frac\partial{\partial\nu}p(t)=0$ on $p(t)$. We first prove a short-time existence theorem for $G$-invariant solutions. Next, we show that $\mathcal F(g(t),p(t))$ is non-decreasing if $(g(t),p(t))$ is such a solution, $\eta(t)$ is independent of $t$, and $\kappa(t)$ is identically~$0$. In the process, we obtain a new formula for the $t$-derivative of $\mathcal F$ under~\eqref{intro_MRF}. The section ends with a discussion of how our results relate to those of~\cite{JL13}.

\section{Short-time existence and uniqueness}\label{sec_STE_uniq}

Consider a compact Lie group $G$ and closed subgroup $H$ of $G$. Suppose the homogeneous space $G/H$ is connected and $(d-1)$-dimensional with $d\ge3$. The objective of this paper is to investigate the Ricci flow on a smooth manifold $M$ diffeomorphic to $[0,1] \times G/H$. It will be convenient for us to assume that
\begin{align*}
M = [0,1] \times G/H.
\end{align*}
Such an assumption does not lead to any loss of generality. Obviously, the manifold $M$ has nonempty boundary $\partial M$ consisting of two connected components, $\{0\}\times G/H$ and $\{1\}\times G/H$. We will use the notation $M_0$ for the interior $M\setminus \partial M$. The group $G$ acts naturally on $M$.

\subsection{The existence and uniqueness theorem}\label{subsec_STE_uniq_results}

The Ricci flow is the partial differential equation
\begin{align}\label{RF_eq}
\frac{\partial}{\partial t}g(t) = -2\Ric(g(t))
\end{align}
for a Riemannian metric $g(t)$ on $M$ depending on the parameter $t\geq 0$. In the right-hand side, $\Ric(g(t))$ stands for the Ricci curvature of $g(t)$. As we explained in the introduction, one may learn about the history, the intuitive meaning, the technical peculiarities and the geometric applications of equation~\eqref{RF_eq} from many books, such as \cite{BCPLLN06,PT06,JMGT07}. 

Suppose $T^\ast\partial M \hat\otimes T^\ast\partial M$ is the bundle of symmetric $(0,2)$-tensors on $\partial M$. Consider a smooth
map
\begin{align*}
F:[0,\infty) \times (T^\ast\partial M \hat\otimes T^\ast\partial M) \to T^\ast\partial M \hat\otimes T^\ast\partial M
\end{align*}
such that $F(t,\cdot)$ is fiber-preserving for all $t \in [0, \infty)$. We will use $F$ to supplement equation~\eqref{RF_eq} with boundary conditions. Before we can do so, however, we need to make some preparations. Namely, suppose $\Gamma(T^\ast\partial M \hat\otimes T^\ast\partial M)$ is the space of continuous sections of $T^\ast\partial M \hat\otimes T^\ast\partial M$. Observe that $F$ induces a map from $[0,\infty)\times \Gamma(T^\ast\partial M \hat\otimes T^\ast\partial M)$ to $\Gamma(T^\ast\partial M \hat\otimes T^\ast\partial M)$. It will be convenient for us to use the same letter $F$ for this map. We assume the images of $G$-invariant sections of $T^\ast\partial M \hat\otimes T^\ast\partial M$ under $F(t,\cdot)$ are themselves $G$-invariant for all $t$. 

Let $\II(g(t))$ be the second fundamental form of $\partial M$ computed in $g(t)$ with respect to the outward unit normal. Thus, $\II(g(t))$ is a $t$-dependent (0,2)-tensor field on $\partial M$. Our sign convention is such that $\II(g(t))$ is positive-definite when $M$ is a closed ball in $\mathbb R^3$ and $g(t)$ is Euclidean. Suppose $g_{\partial M}(t)$ is the Riemmanian metric induced on $\partial M$ by $g(t)$. In this section, we study the Ricci flow equation~\eqref{RF_eq} under the boundary condition
\begin{align}\label{My_BC}
\II(g(t)) = F(t, g_{\partial M}(t)).
\end{align} 
Note that Y. Shen's works~\cite{YS92,YS96} investigated the situation where $F(t, g_{\partial M}(t)) = \lambda g_{\partial M}(t)$ for some $\lambda \in \mathbb{R}$. The arguments from~\cite{YS92,YS96} also apply when $\lambda$ is allowed to depend on~$t$.

Fix a smooth $G$-invariant Riemannian metric $\hat g$ on $M$. We supplement the Ricci flow equation~\eqref{RF_eq} with the initial condition
\begin{align}\label{RF_ini}
g(0) = \hat g.
\end{align}
Our objective in this section is to prove the short-time existence
and the uniqueness of solutions to
problem~\eqref{RF_eq}--\eqref{My_BC}--\eqref{RF_ini}. Before we can state our result, however, we need to impose an assumption on the homogeneous space $G/H$.

Let $\mathfrak g$ be the Lie algebra of the group $G$. Pick an
$\Ad(G)$-invariant scalar product $Q$ on $\mathfrak g$. Suppose $\mathfrak p$ is the orthogonal complement of the Lie algebra of $H$ in $\mathfrak g$ with respect to $Q$. We standardly identify $\mathfrak p$ with the tangent space of $G/H$ at $H$. The isotropy representation of $G/H$ then yields the structure of an $H$-module on $\mathfrak p$. We assume the following property of $\mathfrak p$ throughout the paper.
\begin{hypothesis}\label{assum_decomp_p}
The $H$-module $\mathfrak p$ appears as an orthogonal sum
\begin{align}\label{p_decomp}
\mathfrak p=\mathfrak p_1\oplus\cdots\oplus\mathfrak p_n
\end{align}
of pairwise non-isomorphic irreducible $H$-modules $\mathfrak
p_1,\ldots,\mathfrak p_n$.
\end{hypothesis}
Hypothesis~\ref{assum_decomp_p} is rather standard. It has come up in a number of papers, such as \cite{ADMW00,ADMW11}. Roughly speaking, it ensures that $G$-invariant (0,2)-tensor fields on $G/H$ are diagonal. Indeed, every such tensor field is determined by its restriction to $\mathfrak p$. Because the summands $\mathfrak p_1, \dots, \mathfrak p_n$ in \eqref{p_decomp} are non-isomorphic and irreducible, this restriction must be diagonal with respect to \eqref{p_decomp}. For a slightly more detailed discussion of Hypothesis~\ref{assum_decomp_p}, including a possible alternative to it, see the author's work \cite{AP13}. 

\begin{theorem}
\label{thm_STE}
Suppose
\begin{align}\label{compat_R_zero}
\II(\hat g)=F(0,\hat g_{\partial M}).
\end{align}
For some number
$T>0$, there exists $g:M\times[0,T)\to T^*M\otimes
T^*M$ such that the following statements hold:
\begin{enumerate}
\item
The map $g$ is smooth on $M_0\times(0,T)$ and continuous on
$M\times[0,T)$. Suppose $X$ and $Y$ are $G$-invariant vector fields on $M$ tangent to $\{r\}\times G/H$ for each $r\in[0,1]$. If $X$ and $Y$ are smooth, then the derivative of the map $M\times[0,T)\ni(x,t)\mapsto g(x,t)(X,Y)\in\mathbb R$ in the variable $x$ exits and is continuous on $M\times[0,T)$.

\item
For every $x\in M$ and $t\in[0,T)$, the tensor $g(x,t)$ is a
symmetric positive-definite tensor at the point~$x$. Thus, $g(\cdot, t)$ is
a Riemannian metric on $M$. We will use the notation $g(t)$ for this metric.

\item
The equality $g(t)=\gamma^*g(t)$ holds for all $\gamma\in G$ and
$t\in[0,T)$. In other words, $g(t)$ is
$G$-invariant.

\item
The $t$-dependent Riemannian metric $g(t)$ solves equation~\eqref{RF_eq} on
$M_0\times(0,T)$. This metric satisfies the boundary
condition~\eqref{My_BC} on $\partial M\times(0,T)$ and the initial
condition~\eqref{RF_ini} on $M$.
\end{enumerate}
If, for some number $T>0$, two smooth maps $g_1,g_2:M\times[0,T)\to
T^*M\otimes T^*M$ possess the above
properties~2, 3 and 4, then $g_1=g_2$.
\end{theorem}

\begin{remark}
One may be able to improve the differentiability of $g$ near $\partial M\times[0,T)$ by imposing higher-order compatibility conditions on $F$ and $\hat g$; cf. \cite[Section~5]{PG12}. We will not discuss this further in the present paper.
\end{remark}

\begin{remark}
It is convenient for us to assume that the maps $g_1$ and $g_2$ are smooth on $M\times[0,T)$. However, this assumption can be relaxed. It suffices to demand, for example, that the following two statements hold:
\begin{enumerate}
\item
The maps $g_1$ and $g_2$ are smooth on $M_0\times(0,T)$.
\item
The derivatives $\frac{\partial^i}{\partial t^i}\hat\nabla^j g_1$ and $\frac{\partial^i}{\partial t^i}\hat\nabla^j g_2$ exist and are continuous on $M\times[0,T)$ when $2i+j\le3$.
\end{enumerate}
Here, $\hat\nabla$ denotes covariant differentiation with respect to $\hat g$. The details are left to the reader; cf.~\cite[Theorem~1.3]{PG12}.
\end{remark}

\subsection{Three lemmas}
Before we can prove Theorem~\ref{thm_STE}, we need to make some preparations and state three lemmas. Note that the material laid out here will also be essential to the arguments in Section~\ref{sec_MotFF}. Let us begin by fixing a geodesic $\alpha:[0,1] \to M$ with respect to the metric $\hat g$. We choose $\alpha$ so that it intersects all the $G$-orbits on $M$ orthogonally and $\alpha(r)$ lies in $\{r\}\times G/H$ for all $r \in [0,1]$. The map $\Theta:M \to M$ given by the formula $\Theta(r, \gamma H) = \gamma \alpha(r)$ is a diffeomorphism. The equality
\begin{align*}
\Theta^*\hat g = \hat h^2(r)\,dr\otimes d r + \hat g^r,\qquad r \in[0,1],
\end{align*}
holds true. In the right hand side, $\hat h: [0,1] \to (0, \infty)$ is a smooth function, and $\hat g^r$ is a $G$-invariant Riemannian metric on $G/H$ for every $r \in [0,1]$. Note that $\hat g^r$ is fully determined by its restriction to $\mathfrak p$. Hypothesis~\ref{assum_decomp_p} implies the existence of smooth functions $\hat f_1, \dots, \hat f_n:[0,1] \to (0,\infty)$ such that
\begin{align*}
\hat g^r(X,Y) = \hat f_1^2(r)Q(\pr_{\mathfrak p_1}X, \pr_{\mathfrak p_1}Y)+\dots+\hat f_n^2(r)Q(\pr_{\mathfrak p_n}X, \pr_{\mathfrak p_n}Y),\qquad X,Y \in \mathfrak p.
\end{align*}
The notation $\pr_{\mathfrak p_k}X$ and $\pr_{\mathfrak p_k}Y$ here stands for the projection of $X$ and $Y$ onto $\mathfrak p_k$ for $k = 1, \dots, n$. In what follows, we assume that the diffeomorphism $\Theta$ is the identity map on $M$. This assumption does not lead to loss of generality. Thus, the equality
\begin{align}\label{symm_ans}
\hat g = \hat h^2(r)\,dr\otimes dr + \hat g^r, \qquad r \in [0,1],
\end{align}
holds true. 

Our first lemma essentially shows that any $G$-invariant solution to \eqref{RF_eq}, subject to the initial condition~\eqref{RF_ini}, must have the form~\eqref{symm_ans}. This fact is crucial to the proof of Theorem~\ref{thm_STE}. It is also important to the arguments in Section~\ref{sec_MotFF}.
\begin{lemma}\label{lem_symm_pres}
Assume $w:M_0\times [0,T) \to T^*M\otimes T^*M$ is a smooth map satisfying the following requirements:
\begin{enumerate}
\item For every $x \in M$ and $t \in [0,T)$, the tensor $w(x,t)$ is a symmetric positive-definite tensor at the point~$x$.
\item Given $\gamma \in G$ and $t \in [0,T)$, the Riemannian metric $w(t) = w(\cdot,t)$ satisfies the formula $w(t) = \gamma^*w(t)$.
\item The equality
\begin{align*}
\frac{\partial}{\partial t}w(t) = -2\Ric(w(t))
\end{align*}
holds on $M_0 \times (0,T)$, and the equality 
\begin{align*}
w(0) = \hat g
\end{align*}
holds on $M_0$.
\end{enumerate}
Then
\begin{align*}
w(t) = z^2(r,t)\,dr\otimes dr + w^r(t), \qquad r\in(0,1),\,t\in[0,T).
\end{align*}
In the right-hand side, $z$ is a function on $(0,1)\times[0,T)$ with positive values, and $w^r(t)$ is a $G$-invariant Riemannian metric on $G/H$ for each $r\in(0,1)$ and $t\in[0,T)$. 
\end{lemma}
\begin{remark}
Let us emphasize that Lemma~\ref{lem_symm_pres} does not require any boundary conditions on $w(t)$. 
\end{remark}
\begin{proof}
Fix $r_0 \in (0,1)$. Given $t \in [0,T)$, suppose $v(t)$ is a unit normal to $\{r_0\}\times G/H$ at $(r_0, H)$ with respect to the metric $w(t)$. We assume $v(t)$ points in the direction of $\{1\}\times G/H$. Our plan is to show that $v(t)$ is a scalar multiple of $v(0)$. The assertion of the lemma will follow immediately.

Let $(y_1, \dots, y_d)$ be a local coordinate system on $M$ centred at $(r_0, H)$. Assume that, at $(r_0,H)$, the vectors $\frac{\partial}{\partial y_i}$ are tangent to $\{r_0\}\times G/H$ for $i=1,\dots,d-1$, and $\frac{\partial}{\partial y_d}$ coincides with $v(t)$. The formula
\begin{align*}
v(\tau) = \sum_{i=1}^d \frac{w^{id}(\tau)}{(w^{dd}(\tau))^\frac{1}{2}}\frac{\partial}{\partial y_i}, \qquad \tau \in [0,T),
\end{align*}
holds true. In the right-hand side, $w^{id}(\tau)$ are the components of the inverse of $w(\tau)$ at $(r_0, H)$ in the coordinates $(y_1, \dots, y_d)$. Taking advantage of assumption~3, we find
\begin{align*}
\frac{d}{d\tau}v(\tau)|_{\tau=t} &= \sum_{i,j,l = 1}^d \left(\frac{2W_{jl}(t)w^{ji}(t)w^{dl}(t)}{(w^{dd}(t))^\frac{1}{2}}-\frac{W_{jl}(t)w^{jd}(t)w^{dl}(t)w^{id}(t)}{(w^{dd}(t))^\frac{3}{2}}\right)\frac{\partial}{\partial y_i}\\
&=-\sum_{i=1}^d\left(W_{dd}(t)w^{id}(t)(w^{dd}(t))^\frac{1}{2} - 2\sum_{j=1}^dW_{jd}(t)w^{ji}(t)(w^{dd}(t))^\frac{1}{2}\right)\frac{\partial}{\partial y_i}\\
&=W_{dd}(t)(w^{dd}(t))^\frac{3}{2}\frac{\partial}{\partial y_d} = W_{dd}(t)w^{dd}(t)v(t)\\
&=\Ric(w(t))(v(t),v(t))v(t).
\end{align*}
Here, we write $W_{jl}(t)$ for the components of $\Ric(w(t))$ at $(r_0, H)$. To pass from the second line to the third, we used that fact that $W_{id}(t) = 0$ for all $i = 1, \dots, d-1$. This follows from Hypothesis~\ref{assum_decomp_p} (see~\cite[Proposition~1.14]{KGWZ02}).

The above equalities imply
\begin{align*}
v(t) = v(0)\exp\left(\int_0^t\Ric(w(\tau))(v(\tau),v(\tau))\,d\tau\right).
\end{align*}
Consequently, $v(t)$ is a scalar multiple of $v(0)$, and the assertion of the lemma becomes evident. 
\end{proof}

Given $T>0$, suppose $h,f_1,\dots,f_n$ are functions acting from $[0,1]\times[0,T)$ to $(0,\infty)$. Assume these functions are smooth on $(0,1)\times[0,T)$, $h$ is continuous on $[0,1]\times[0,T)$, and $f_1,\dots,f_n$ have first derivatives in $r$ that are continuous on $[0,1]\times[0,T)$. For every $t \in [0,T)$, define a Riemannian metric $g(t)$ on $M$ by setting 
\begin{align}\label{symm_met}
g(t) = h^2(r,t)\,dr\otimes dr + g^r(t), \qquad r\in[0,1].
\end{align}
This formula is analogous to~\eqref{symm_ans}. The notation $g^r(t)$ stands for the $G$-invariant Riemannian metric on $G/H$ such that 
\begin{align}\label{symm_met_pr}
g^r(X,Y) = f_1^2(r,t)Q(\pr_{\mathfrak p_1}X, \pr_{\mathfrak p_1}Y)+\dots+f_n^2(r,t)Q(\pr_{\mathfrak p_n}X, \pr_{\mathfrak p_n}Y),\qquad X,Y \in \mathfrak p.
\end{align}
We will demonstrate that it is possible to choose $T$ and $h, f_1,\dots,f_n$ in such a way that $g(t)$ solves the initial-boundary-value problem~\eqref{RF_eq}--\eqref{My_BC}--\eqref{RF_ini}. This will prove the existence portion of Theorem~\ref{thm_STE}.

Our second lemma provides an expression of the Ricci curvature of the metric $g(t)$ in terms of the functions $h, f_1,\dots,f_n$. To formulate it, we need more notation. Let $[\cdot,\cdot]$ and $P$ be the Lie bracket and the Killing form of the Lie algebra $\mathfrak g$. The irreducibility of the summands in the decomposition~\eqref{p_decomp} implies the existence of nonnegative numbers $\beta_1,\dots,\beta_n$ such that 
\begin{align*}
P(X,Y) = -\beta_kQ(X,Y), \qquad k =1,\dots,n,\;X,Y\in\mathfrak p_k.
\end{align*} 
Because the group $G$ is compact and Hypothesis~\ref{assum_decomp_p} holds, at least one of these numbers must be strictly positive. Let $d_k$ be the dimension of $\mathfrak p_k$.
We choose a $Q$-orthonormal basis $(e_j)_{j=1}^{d-1}$ of the space $\mathfrak p$ adapted to~\eqref{p_decomp}. In addition to $\beta_1,\dots,\beta_n$, we define 
\begin{align*}
\gamma_{i,k}^l = \frac{1}{d_i}\sum Q([e_{\iota_i},e_{\iota_k}], e_{\iota_l})^2
\end{align*}
for $i,k,l = 1, \dots, n$. The sum here is taken over all $\iota_i, \iota_k$ and $\iota_l$ such that $\iota_i \in \mathfrak p_i$, $\iota_k \in \mathfrak p_k$ and $\iota_l \in \mathfrak p_l$. Note that $\gamma_{i,k}^l$ is independent of the choice of $(e_j)_{j=1}^{d-1}$. One easily checks that $d_i\gamma_{i,k}^l = d_k\gamma_{k,i}^l = d_l\gamma_{l,i}^k$. For more identities satisfied by $(\gamma_{i,k}^l)_{i,k,l=1}^n$, consult~\cite{KGWZ02} and references therein. 

\begin{lemma}\label{lem_red_Ric}
The Ricci curvature of the Riemannian metric $g(t)$ given by~\eqref{symm_met}--\eqref{symm_met_pr} obeys the equality
\begin{align*}
\Ric(g(t)) = -\sum_{k=1}^n d_k\left(\frac{f_{krr}}{f_k}-\frac{h_r f_{kr}}{hf_k}\right)dr\otimes dr + \Ric^r(g(t)),\qquad r\in(0,1),
\end{align*}
where $\Ric^r(g(t))$ is the $G$-invariant $(0,2)$-tensor field on $G/H$ such that
\begin{align*}
\Ric^r&(g(t))(X,Y)\\
&=\sum_{i=1}^n\left(\frac{\beta_i}{2} + \sum_{k,l=1}^n\gamma_{i,k}^l\frac{f_i^4 - 2f_k^4}{4f_k^2f_l^2} - \frac{f_if_{ir}}{h}\sum_{k=1}^n d_k\frac{f_{kr}}{hf_k} + \frac{f_{ir}^2}{h^2}-\frac{f_if_{irr}}{h^2}+\frac{f_ih_rf_{ir}}{h^3}\right)Q(\pr_{\mathfrak p_i}X, \pr_{\mathfrak p_i}Y)
\end{align*}
for $X,Y \in \mathfrak p$. The subscript $r$ here means differentiation in $r \in (0,1)$.
\end{lemma}
The reader will find the proof of Lemma~\ref{lem_red_Ric} in~\cite{AP13}. The computations were essentially made in~\cite{KGWZ02}.

Lastly, we need an existence and uniqueness result for parabolic systems of partial differential equations on the interval $[0,1]$ under non-homogeneous Neumann boundary conditions. Given $T>0$, we will deal with the Sobolev-type space $W_5^{2,1}((0,1)\times(0,T))$. The reader may see~\cite[Section~I.1]{OLVSNU68} for its rigorous definition. It will be convenient for us to abbreviate $W_5^{2,1}((0,1)\times(0,T))$ to $W_{5,T}^{2,1}$. Throughout the paper, we will also encounter H\"older-type spaces $H^{\delta,\frac\delta2}([0,1]\times[0,T])$ with $\delta>0$. We refer to~\cite[Sections~I.1]{OLVSNU68} for their precise definition. To keep our notation short, we will abbreviate $H^{\delta,\frac\delta2}([0,1]\times[0,T])$ to $H_T^{\delta,\frac\delta2}$. Note that, according to~\cite[Lemma~3.3 in Chapter~II]{OLVSNU68}, every function in $W_{5,T}^{2,1}$ must lie in $H_T^{\frac65,\frac35}$. In particular, the derivative of such a function with respect to the first variable is continuous on $[0,1]\times[0,T]$. This fact is crucial to our further arguments.

\begin{lemma}\label{lem_11STE}
Fix $m \in \mathbb{N}$ and consider smooth functions
\begin{align*}
a&:[0,1]\times[0,\infty)\times\mathbb{R}^m\to (0,\infty),\\
A_i&:[0,1]\times[0,\infty)\times\mathbb{R}^m\times\mathbb{R}^m\to \mathbb{R},\\
B_i&:\{0,1\}\times[0,\infty)\times\mathbb{R}^m\to\mathbb{R},\\
\hat v_i&:[0,1] \to \mathbb R,\qquad i=1,\dots,m.
\end{align*}
Assume that
\begin{align*}
\hat v_{ir}(j)=B_i(j,0,\hat v(j)),\qquad j=0,1,\;i=1,\dots,m,
\end{align*}
where $\hat v=(\hat v_1,\ldots,\hat v_m)$. For some $T>0$, there exist
\begin{align*}
v_i:[0,1]\times[0,T]\to \mathbb{R},\qquad i=1,\dots,m,
\end{align*}
satisfying the following statements:
\begin{enumerate}
\item Each $v_i$ is smooth on $(0,1)\times[0,T]$ and lies in the space $W_{5,T}^{2,1}$.
\item For every $i=1,\dots,m$, the function $v_i$ solves the equation
\begin{align}\label{11_eq}
v_{it}(r,t) = a(r,t,v(r,t))v_{irr}(r,t)+A_i(r,t,v(r,t),v_r(r,t)),\qquad r\in(0,1),\;t\in(0,T),
\end{align}
subject to the boundary conditions
\begin{align*}
v_{ir}(j,t)=B_i(j,t,v(j,t)),\qquad j=0,1,\;t\in(0,T),
\end{align*}
and the initial condition
\begin{align*}
v_i(r,0) = \hat v_i(r),\qquad r\in[0,1].
\end{align*}
\end{enumerate}
Here, $v$ stands for $(v_1,\ldots,v_m)$, subscript $t$ denotes the derivative in $t$, and $v_r$ denotes the component-wise derivative in $r$.
If, for some $T>0$, the arrays
\begin{align*}
v_{1,i}&:[0,1]\times[0,T]\to\mathbb{R},\\
v_{2,i}&:[0,1]\times[0,T]\to\mathbb{R},\qquad i=1,\dots,m,
\end{align*}
possess the above properties 1 and 2, then
\begin{align*}
v_{1,i}=v_{2,i},\qquad i=1,\dots,m.
\end{align*}
\end{lemma}

One may establish Lemma~\ref{lem_11STE} by repeating the reasoning from the proof of~\cite[Theorem 2.1]{AP10} and from~\cite[Remark 2.2 (i)]{PW91} with minor adjustments. We will not discuss these adjustments here, as they are fairly straightforward. The need for them arises primarily because the boundary of $[0,1]$ is 0-dimensional and, as a consequence, the definitions of Sobolev-type spaces on $\partial [0,1]\times[0,T]$
 require clarification.

The result in~\cite{AP10} is essentially the existence portion of Lemma~\ref{lem_11STE} with the interval $[0,1]$ replaced by a Riemannian manifold of dimension two or higher. The method of proof employed in~\cite{AP10} relies on a fixed-point argument, as executed by Weidemaier in the proof of~\cite[Theorem 2.1]{PW91}. A solution to the initial-boundary-value problem is first constructed in a Sobolev-type space. Its regularity is then established via a bootstrapping argument. Classical results from~\cite{OLVSNU68} are used in the process. The uniqueness portion of Lemma~\ref{lem_11STE} follows from the arguments in~\cite[Remark~2.2~(i)]{PW91}. Note that the reader may find results closely related to the lemma in Amann's paper~\cite{HA90}; specifically, see the theorem in the introduction.  

\subsection{The argument for existence and uniqueness}\label{sec_pfSTE}
According to Lemma~\ref{lem_red_Ric}, the Riemannian metric $g(t)$ satisfies the Ricci flow equation~\eqref{RF_eq} if
\begin{align}\label{red_RF_eq}
h_t &= \sum_{k=1}^nd_k\left(\frac{f_{krr}}{hf_k} - \frac{h_rf_{kr}}{h^2f_k}\right),\notag\\
f_{it} &= -\frac{\beta_i}{2f_i} - \sum_{k,l=1}^n\gamma_{i,k}^l\frac{f_i^4-2f_k^4}{4f_if_k^2f_l^2}+\frac{f_{ir}}{h}\sum_{k=1}^nd_k\frac{f_{kr}}{hf_k}-\frac{f_{ir}^2}{h^2f_i}+\frac{f_{irr}}{h^2}-\frac{h_rf_{ir}}{h^3},\notag\\
r &\in (0,1),\;t\in(0,T),\;i=1,\dots,n.
\end{align}
Let us now write the boundary condition~\eqref{My_BC} and the initial condition~\eqref{RF_ini} in terms of $h,f_1,\dots,f_n$. Given a $G$-invariant section $u \in \Gamma(T^*\partial M\hat\otimes T^*\partial M)$ and $j = 0,1,$ the restriction of $u$ to $\{j\}\times G/H$ is fully determined by how $u$ acts on the tangent space $T_{\{ j\}\times H}(\{j\}\times G/H)$. Identifying this space with $\mathfrak p$ in the natural way, we define the numbers $u_{j,1},\dots,u_{j,n} \in \mathbb{R}$ by the equality
\begin{align*}
u|_{\{j\}\times G/H}(X,Y)= u_{j,1}Q(\pr_{\mathfrak p_1}X, \pr_{\mathfrak p_1}Y)+\dots+u_{j,n}Q(\pr_{\mathfrak p_n}X, \pr_{\mathfrak p_n}Y),\qquad X,Y \in \mathfrak p.
\end{align*}
There exist smooth functions $F_{j,1},\dots,F_{j,n}$ from $[0,\infty)\times\mathbb{R}^n$ to $\mathbb{R}$ such that, for all $t\in[0,\infty)$ and all $G$-invariant $u \in \Gamma(T^*\partial M\hat\otimes T^*\partial M)$, we have
\begin{align*}
F(t,u)|_{\{j\}\times G/H}(X,Y)= &\;F_{j,1}(t,u_{j,1},\dots,u_{j,n})Q(\pr_{\mathfrak p_1}X, \pr_{\mathfrak p_1}Y)\\
&+\dots+F_{j,n}(t,u_{j,1},\dots,u_{j,n})Q(\pr_{\mathfrak p_n}X, \pr_{\mathfrak p_n}Y),\qquad X,Y \in \mathfrak p.
\end{align*}
A computation (cf.~\cite[Lemma 2]{AP11}) shows that $g(t)$ obeys~\eqref{My_BC} when
\begin{align}\label{red_BC}
f_{ir}(j,t)=(-1)^{j+1}\frac{h(j,t)F_{j,i}(t,f_1^2(j,t),\dots,f_n^2(j,t))}{f_i(j,t)},\qquad j=0,1,~t\in(0,T),~i=1,\ldots,n.
\end{align}
It is also easy to see that~\eqref{RF_ini} holds when
\begin{align}\label{red_ini}
h(r,0)=\hat h(r),\qquad f_i(r,0) =\hat f_i(r),\qquad r\in[0,1],\;i=1,\dots,n.
\end{align}
\begin{proof}[Proof of Theorem~\ref{thm_STE}]
We first prove the existence of a map $g$ possessing the listed properties. The method we use is an adaptation of the DeTurck trick (to be specific, the version of the DeTurck trick described in~\cite[Section~2.6]{BCPLLN06}). The main idea is to replace~\eqref{red_RF_eq} with the more tractable system~\eqref{par_red_RF}. Lemma~\ref{lem_11STE} will guarantee the short-time existence of a solution to~\eqref{par_red_RF} under appropriate boundary and initial conditions. Modifying this solution, we will produce functions $h, f_1,\dots,f_n$ which satisfy~\eqref{red_RF_eq}--\eqref{red_BC}--\eqref{red_ini}. The $G$-invariant $t$-dependent metric given by~\eqref{symm_met} and \eqref{symm_met_pr} will define a mapping with the properties asserted in the theorem. In particular, this metric will solve the Ricci flow. 

Lemma~\ref{lem_11STE} yields the existence, for some $T>0$, of functions $\bar h, \bar f_1, \dots, \bar f_n:[0,1]\times[0,T]\to(0,\infty)$ satisfying the equations
\begin{align}\label{par_red_RF}
\bar h_t &= \frac{\bar h_{rr}}{\bar h^2}-\frac{2\bar h_r^2}{\bar h^3}+\sum_{k=1}^nd_k\frac{\bar f_{kr}^2}{\bar h \bar f_k^2}-\Biggl(\frac{\hat h_r}{\hat h^2}+\sum_{k=1}^nd_k\frac{\hat f_{kr}}{\hat h\hat f_k}\Biggr)_r,\notag\\
\bar f_{it} &= \frac{\bar f_{irr}}{\bar h^2}-\frac{\bar f_{ir}^2}{\bar h^2 \bar f_i}-\sum_{k,l=1}^n\gamma_{i,k}^l\frac{\bar f_i^4 - 2\bar f_k^4}{4\bar f_i\bar f_k^2 \bar f_l^2} - \frac{\beta_i}{2\bar f_i} -\frac{\hat h_r\bar f_{ir}}{\bar h\hat h^2}+\sum_{k=1}^nd_k\frac{\bar f_{ir}\hat f_{kr}}{\bar h\hat h\hat f_k},\notag\\ r&\in(0,1),\;t\in(0,T),\;i=1,\dots,n,
\end{align}
the boundary conditions
\begin{align}\label{par_red_BC}
\bar h_r(j,t) &= (-1)^{j+1}\sum_{k=1}^nd_k\frac{\bar h^2(j,t)F_{j,k}(t,\bar f_1^2(j,t),\dots, \bar f_n^2(j,t))}{\bar f_k^2(j,t)}+\frac{\bar h^2(j,t)}{\hat h(j)}\Biggl(\frac{\hat h_r(j)}{\hat h(j)}-\sum_{k=1}^nd_k\frac{\hat f_{kr}(j)}{\hat f_k(j)}\Biggr),\notag\\
\bar f_{ir}(j,t) &= (-1)^{j+1}\frac{\bar h(j,t)F_{j,i}(t,\bar f_1^2(j,t),\dots, \bar f_n^2(j,t))}{\bar f_i(j,t)},\qquad j=0,1,\;t\in(0,T),\;i=1,\dots,n,
\end{align}
and the initial conditions
\begin{align}\label{par_red_ini}
\bar h(r,0) = \hat h(r),\qquad \bar f_i(r,0) = \hat f_i(r),\qquad r\in[0,1],\;i=1,\dots,n.
\end{align}
These functions are smooth on $(0,1)\times[0,T]$. They lie in the space $W_{5,T}^{2,1}$ and, consequently, the space $H_T^{\frac65,\frac35}$. 
Equality~\eqref{compat_R_zero} and classical regularity results (see, e.g.,~\cite[Theorem~5.3 in Chapter~IV]{OLVSNU68}) imply that the derivatives $\bar h_{rr}, \bar f_{1rr},\dots,\bar f_{nrr}$ exist and are continuous on $[0,1]\times[0,T]$. This enables us to define a function $\phi:[0,1]\times[0,T)\to\mathbb{R}$ by the formula
\begin{align}\label{def_phi}
\phi_t(r,t) = -\Bigg(\frac{\bar h_\rho(\rho,t)}{\bar h^3(\rho,t)}-\sum_{k=1}^nd_k\frac{\bar f_{k\rho}(\rho,t)}{\bar h^2(\rho,t)\bar f_k(\rho,t)}-\frac{\hat h_\rho(\rho)}{\bar h(\rho,t)\hat h^2(\rho)}&+\sum_{k=1}^nd_k\frac{\hat f_{k\rho}(\rho)}{\bar h(\rho,t)\hat h(\rho)\hat f_k(\rho)}\Bigg)\Bigg|_{\rho = \phi(r,t)}, \notag \\ r&\in[0,1],\;t\in(0,T),
\end{align}
and the requirement that $\phi(r,0) = r$ when $r \in [0,1]$. Here, the subscript $\rho$ denotes differentiation in $\rho$. We will use $\phi$ to modify $\bar h, \bar f_1,\dots,\bar f_n$ and obtain a solution to~\eqref{red_RF_eq}--\eqref{red_BC}--\eqref{red_ini}. Remark~\ref{rem_phimotiv} will explain the thought process that lead us to considering system~\eqref{par_red_RF} and to~\eqref{def_phi}.

The function $\phi$ is smooth on $(0,1)\times[0,T)$. Its derivative with respect to $r$ exists and is continuous on $[0,1]\times[0,T)$. Moreover, this derivative is strictly positive on $[0,1]\times[0,T)$. The boundary conditions~\eqref{par_red_BC} imply that the right-hand side of~\eqref{def_phi} is 0 when $\phi(r,t)\in\{0,1\}$ and $t \in (0,T)$. Consequently, the equality
\begin{align}\label{phi_bdy}
\phi(j,t)=j,\qquad j = 0,1,\;t\in[0,T),
\end{align}
holds true. Consider the real-valued functions $h, f_1, \dots, f_n$ on $[0,1]\times[0,T)$ defined as
\begin{align}\label{def_h_f}
h(r,t) = \phi_r(r,t)\bar h(\phi(r,t),t),\qquad f_i(r,t)=\bar f_i(\phi(r,t),t),\qquad r\in[0,1],\;t\in[0,T),\;i=1,\dots,n.
\end{align}
Remark~\ref{rem_phimotiv} discusses the geometric meaning of~\eqref{def_h_f}. Keeping in mind that $\bar h, \bar f_1,\dots, \bar f_n$ satisfy~\eqref{par_red_RF}, we can verify by direct computation that $h,f_1,\dots,f_n$ satisfy~\eqref{red_RF_eq}. One performs a computation of similar nature when one carries out the DeTurck trick on closed manifolds; cf.~\cite[Section~2.6]{BCPLLN06} and~\cite[Section~5.2]{PT06}. The boundary conditions~\eqref{par_red_BC} and formulas~\eqref{phi_bdy} imply that
\begin{align*}
f_{ir}(j,t) &= (-1)^{j+1}\phi_r(j,t)\frac{\bar h(j,t)F_{j,i}(t,\bar f_1^2(j,t),\dots, \bar f_n^2(j,t))}{\bar f_i(j,t)}\\
&= (-1)^{j+1}\frac{h(j,t)F_{j,i}(t,f_1^2(j,t),\dots, f_n^2(j,t))}{f_i(j,t)},\qquad j=0,1,\; t\in(0,T),\; i=1,\dots,n.
\end{align*}
Therefore, $h,f_1,\dots,f_n$ satisfy~\eqref{red_BC}. Finally, equality~\eqref{red_ini} holds for $h,f_1,\dots,f_n$ because~\eqref{par_red_ini} holds for $\bar h, \bar f_1,\dots,\bar f_n$ and $\phi(\cdot,0)$ is the identity map on $[0,1]$. 

We define a $G$-invariant $t$-dependent metric $g(t)$ on $M$ through formulas~\eqref{symm_met} and~\eqref{symm_met_pr}. One may interpret $g(t)$ as a map from $M\times[0,T)$ to $T^*M\otimes T^*M$. This map obviously possesses the desired properties.

Let us prove the uniqueness portion of the theorem. To do so, we first rewrite the Ricci flow equations for $g_1$ and $g_2$ in the form~\eqref{red_RF_eq}. We then replace the obtained systems with more tractable systems analogous to~\eqref{par_red_RF}. The approach we take is rather classical; cf.~\cite[Section~2.6]{BCPLLN06}. To make it work in our setting, however, we need an auxiliary result (specifically, Lemma~\ref{lem_symm_pres} above). The purpose of this result is to help us demonstrate that $g_1$ and $g_2$ can be simultaneously diagonalized. Roughly speaking, it states that the normals to $G$-orbits on $M$ with respect to $g_1$ and $g_2$ point in the same direction.

We argue by contradiction. Suppose $g_1$ and $g_2$ do not coincide. Without loss of generality, assume that 
\begin{align*}
\sup\{\tau \in [0,T)\,|\,g_1 = g_2\text{ on }M\times[0,\tau]\}=0.
\end{align*} 
Lemma~\ref{lem_symm_pres} and Hypothesis~\ref{assum_decomp_p} imply the existence of smooth positive functions $h_1,f_{1,1},\dots,f_{1,n}$ and $h_2,\allowbreak f_{2,1},\dots,f_{2,n}$ on $[0,1]\times[0,T)$ satisfying the formula 
\begin{align*}
g_m(t) = h_m^2(r,t)\,dr\otimes dr + g_m^r(t),\qquad m=1,2,\;r\in[0,1],\;t\in[0,T).
\end{align*}
Here, $g_m(t)$ is the $t$-dependent Riemannian metric given by the map $g_m$, and $g_m^r(t)$ is the $G$-invariant Riemannian metric on $G/H$ such that
\begin{align*}
g_m^r(t)(X,Y)= f_{m,1}^2(r,t)Q(\pr_{\mathfrak p_1}X, \pr_{\mathfrak p_1}Y)+\dots+f_{m,n}^2(r,t)Q(\pr_{\mathfrak p_n}X, \pr_{\mathfrak p_n}Y),\qquad X,Y \in \mathfrak p.
\end{align*}
It is clear that formulas~\eqref{red_RF_eq}--\eqref{red_BC}--\eqref{red_ini} will still hold if we replace $h,f_1,\dots,f_n$ in them by $h_m,\allowbreak f_{m,1},\dots,f_{m,n}$ for either $m=1$ or $m=2$.

According to classical existence results for parabolic problems (see, e.g.,~\cite[Theorem~6.1 in Chapter~V]{OLVSNU68}), for some $S \in \big(0,\frac T2\big)$, we can find $\phi_1,\phi_2:[0,1]\times[0,2S)\to\mathbb{R}$ obeying the equation
\begin{align*}
\phi_{mt}(r,t)&=\frac{\phi_{mrr}(r,t)}{h_m^2(r,t)} - \frac{\phi_{mr}(r,t)h_{mr}(r,t)}{h_m^3(r,t)}+\sum_{k=1}^nd_k\frac{\phi_{mr}(r,t)(f_{m,k})_r(r,t)}{h_m^2(r,t)f_{m,k}(r,t)} \\
& \hphantom{=}~+ \frac{\phi_{mr}(r,t)}{h_m(r,t)}\Biggl(\frac{\hat h_{m\rho}(\rho)}{\hat h_m^2(\rho)}-\sum_{k=1}^nd_k\frac{(\hat f_{m,k})_\rho(\rho)}{\hat h_m(\rho)\hat f_{m,k}(\rho)}\Biggr)\Bigg|_{\rho=\phi_m(r,t)}
,\qquad m=1,2,\;r\in(0,1),\;t\in(0,2S),
\end{align*}
the boundary conditions
\begin{align*}
\phi_m(j,t)=j, \qquad m=1,2,\;j=0,1,\;t\in(0,2S),
\end{align*}
and the initial condition
\begin{align*}
\phi_m(r,0)=r,\qquad m=1,2,\;r\in[0,1].
\end{align*}
Given $\epsilon\in(0,1)$, these $\phi_1$ and $\phi_2$ lie in $H_S^{2+\epsilon,1+\frac\epsilon2}$. In fact, they have third derivatives in $r$ that are continuous on $[0,1]\times[0,S]$, and they are smooth on $(0,1)\times[0,S]$.
By choosing $S$ sufficiently small, we ensure that $\phi_{1r}$ and $\phi_{2r}$ are strictly positive on $[0,1]\times[0,S]$. The formula 
\begin{align*}
\phi_m([0,1]\times\{t\})=[0,1],\qquad m=1,2,\;t\in[0,S],
\end{align*}
holds true. The thought process behind introducing $\phi_1$ and $\phi_2$ is explained in Remark~\ref{rem_phi1&2motiv}. 

Let $\phi_m^{-1}(\cdot,t)$ denote the inverse of the map $\phi_m(\cdot,t):[0,1]\to[0,1]$ for each $t \in[0,S]$ and $m=1,2$. Consider the real-valued functions $\bar h_1, \bar f_{1,1},\dots,\bar f_{1,n}$ and $\bar h_2, \bar f_{2,1},\dots,\bar f_{2,n}$ on $[0,1]\times[0,S]$ defined as 
\begin{align*}
\bar h_m(r,t) &= (\phi_m^{-1})_r(r,t)h_m(\phi_m^{-1}(r,t),t),\\
\bar f_{m,i}(r,t) &= f_{m,i}(\phi_m^{-1}(r,t),t),\qquad m=1,2,\;r\in[0,1],\;t\in[0,S],\;i=1,\dots,n.
\end{align*}
It is easy to verify that
\begin{align}\label{phi_mt}
\phi_{mt}(r,t) = &-\Biggl(\frac{\bar h_{m\rho}(\rho,t)}{\bar h_m^3(\rho,t)}-\sum_{k=1}^nd_k\frac{(\bar f_{m,k})_\rho(\rho,t)}{\bar h_m^2(\rho,t)\bar f_{m,k}(\rho,t)}\Biggr)\Bigg|_{\rho = \phi_m(r,t)}\notag \\ &+\Biggl(\frac{\hat h_{m\rho}(\rho)}{\bar h_m(\rho,t)\hat h_m^2(\rho)}-\sum_{k=1}^nd_k\frac{(\hat f_{m,k})_\rho(\rho)}{\bar h_m(\rho,t)\hat h_m(\rho)\hat f_{m,k}(\rho)}\Biggr)\Bigg|_{\rho = \phi_m(r,t)}, \notag\\ &\;m=1,2,\;r\in[0,1],\;t\in[0,S].
\end{align}
Formulas~\eqref{par_red_RF}--\eqref{par_red_BC}--\eqref{par_red_ini} will still hold if we replace $T$ in them by $S$ and $\bar h, \bar f_1,\dots,\bar f_n$ by $\bar h_m, \bar f_{m,1},\dots, \bar f_{m,n}$ for either $m=1$ or $m=2$. Also, $\bar h_1, \bar f_{1,1},\dots, \bar f_{1,n}$ and $\bar h_2, \bar f_{2,1},\dots,\bar f_{2,n}$ have two derivatives in $r$ and one in $t$ that are continuous on $[0,1]\times[0,S]$. 
Lemma~\ref{lem_11STE} tells us that $\bar h_1 = \bar h_2$ and $\bar f_{1,i} = \bar f_{2,i}$ when $i = 1, \dots, n$. Furthermore, according to~\eqref{phi_mt} and the standard uniqueness theorems for ordinary differential equations, $\phi_1$ must coincide with $\phi_2$ on $[0,1]\times[0,S]$. Because
\begin{align*}
h_m(r,t) &= \phi_{mr}(r,t)\bar h_m(\phi_m(r,t),t),\\
f_{m,i}(r,t)&=\bar f_{m,i}(\phi_m(r,t),t),\qquad m=1,2,\;r\in[0,1],\;t\in[0,S],\;i=1,\dots,n,
\end{align*}
it becomes clear that $h_1 = h_2$ and $f_{1,i}=f_{2,i}$ on $[0,1]\times[0,S]$ when $i=1,\dots,n$. This is a contradiction. 
\end{proof}

\begin{remark}\label{rem_phimotiv}
The following principle underlies the DeTurck trick: if the metric $g(t)$ solves the Ricci flow equation on $M_0\times(0,T)$, then for a properly chosen $t$-dependent diffeomorphism $\Phi(\cdot,t)$ of $M_0$, the pullback $(\Phi^{-1})^*(\cdot,t)g(t)$ must solve a more tractable equation. The proof of the existence part of Theorem~\ref{thm_STE} is based on this principle. To clarify, we need to make two observations. 
\begin{enumerate}
\item The function $\phi$ defines a mapping $\Phi:M\times[0,T)\to M$ via the formula 
\begin{align*}
\Phi((r,\gamma H),t) = (\phi(r,t), \gamma H),\qquad r\in[0,1],\;\gamma\in G,\;t\in[0,T).
\end{align*}
Because the derivative of $\phi$ in the first variable is positive on $[0,1]\times[0,T)$, and because~\eqref{phi_bdy} holds, $\Phi(\cdot,t)$ must be a diffeomorphism of $M$ for every $t \in [0,T)$. Assuming $g(t)$ is given by~\eqref{symm_met} and~\eqref{symm_met_pr}, we can easily check that 
\begin{align*}
(\Phi^{-1})^*(\cdot,t)g(t) = \bar h^2(r,t)\,dr\otimes dr + \bar g^r(t),\qquad r\in[0,1],\;t\in[0,T),
\end{align*}
where $\bar g^r(t)$ is the $G$-invariant metric on $G/H$ with
\begin{align*}
\bar g^r(t)(X,Y) = \bar f_1^2(r,t)Q(\pr_{\mathfrak p_1}X, \pr_{\mathfrak p_1}Y)+\dots+\bar f_n^2(r,t)Q(\pr_{\mathfrak p_n}X, \pr_{\mathfrak p_n}Y),\qquad X,Y \in \mathfrak p.
\end{align*}
Here, $\bar h, \bar f_1, \dots, \bar f_n$ obey~\eqref{def_h_f}.
\item If the functions $h,f_1,\dots,f_n$ are to satisfy~\eqref{red_RF_eq}, then $\bar h, \bar f_1,\dots, \bar f_n$ must satisfy
\begin{align*}
\bar h_t(r,t) &= \sum_{k=1}^nd_k\left(\frac{\bar f_{krr}(r,t)}{\bar h(r,t)\bar f_k(r,t)}-\frac{\bar h_r(r,t) \bar f_{kr}(r,t)}{\bar h^2(r,t) \bar f_k(r,t)}\right) \\ &\hphantom{=}~-\bar h_r(r,t)\phi_t(\phi^{-1}(r,t),t) - \bar h(r,t)(\phi_t(\phi^{-1}(r,t),t))_r,\\
\bar f_{it}(r,t) &= \frac{\bar f_{irr}(r,t)}{\bar h^2(r,t)}-\frac{\bar f_{ir}^2(r,t)}{\bar h^2(r,t) \bar f_i(r,t)} - \sum_{k,l=1}^n\gamma_{i,k}^l\frac{\bar f_i^4(r,t) - 2\bar f^4_k(r,t)}{4\bar f_i(r,t) \bar f_k^2(r,t) \bar f_l^2(r,t)} \\ &\hphantom{=}~- \frac{\beta_i}{2\bar f_i(r,t)}-\frac{\bar h_r(r,t) \bar f_{ir}(r,t)}{\bar h^3(r,t)} + \sum_{k=1}^nd_k\frac{\bar f_{ir}(r,t)\bar f_{kr}(r,t)}{\bar h^2(r,t)\bar f_k(r,t)} - \bar f_{ir}(r,t)\phi_t(\phi^{-1}(r,t),t),\\
r&\in(0,1),\;t\in(0,T),\;i=1,\dots,n.
\end{align*}
We define $\phi$ by~\eqref{def_phi} to ensure that $\frac{\bar h_{rr}}{\bar h^2}$ is the only second-order term in the first equation. The above system for $\bar h, \bar f_1, \dots, \bar f_n$ then takes the form~\eqref{11_eq} (in fact, it coincides with~\eqref{par_red_RF}), and Lemma~\ref{lem_11STE} guarantees the existence of a solution. Formulas~\eqref{def_h_f} enable us to obtain the functions $h, f_1, \dots,f_n$ from this solution. 
\end{enumerate}
\end{remark}
\begin{remark}\label{rem_phi1&2motiv}
A simple computation based on~\eqref{def_phi} and~\eqref{def_h_f} yields
\begin{align*}
\phi_t(r,t) = \frac{\phi_{rr}(r,t)}{h^2(r,t)}&-\frac{\phi_r(r,t)h_r(r,t)}{h^3(r,t)}+\sum_{k=1}^nd_k\frac{\phi_r(r,t)f_{kr}(r,t)}{h^2(r,t)f_k(r,t)} \\ &+\frac{\phi_r(r,t)}{h(r,t)}\Biggl(\frac{\hat h_\rho(\rho)}{\hat h^2(\rho)}-\sum_{k=1}^nd_k\frac{\hat f_{k\rho}(\rho)}{\hat h(\rho)\hat f_k(\rho)}\Biggr)\Bigg|_{\rho=\phi(r,t)},\qquad r\in(0,1),\;t\in(0,T).
\end{align*}
This equation motivates the definition of the maps $\phi_1$ and $\phi_2$ in the proof of the theorem. 
\end{remark}

\section{Perelman's $\mathcal F$-functional}\label{sec_MotFF}

Suppose $w$ is a smooth Riemannian metric on $M$ and $q$ is a smooth real-valued function on $M$. By definition, the Perelman $\mathcal F$-functional takes the pair $(w,q)$ to the number
\begin{align*}
\mathcal F(w,q) = \int_M(R(w) + |\nabla q|^2)e^{-q}\,d\mu.
\end{align*}
Here, $R(w)$ is the scalar curvature of $w$. The absolute value and the gradient are taken with respect to~$w$. The letter $\mu$ denotes the $w$-volume measure on $M$. The purpose of this section is to relate the Ricci flow on $M$ to the functional $\mathcal F$ and its monotonicity properties. The main challenge, of course, lies in the nonemptiness of $\partial M$.

\subsection{The modified Ricci flow}

Fix a smooth Riemannian metric $\eta(t)$ and a smooth real-valued function $\kappa(t)$ on $\partial M$ depending on a parameter $t \in [0,\infty)$. For some $T>0$ and $\delta>0$, 
suppose $\tilde g(t)$ is a $G$-invariant solution to the equation 
\begin{align}\label{tilde_RF_eq}
\frac{\partial}{\partial t}\tilde g(t) = -2\Ric(\tilde g(t))
\end{align}
on $M_0\times(0,T+\delta)$ subject to the boundary conditions
\begin{align}\label{tilde_RF_BC}
[\tilde g_{\partial M}(t)] = [\eta(t)],\qquad \mathcal H(\tilde g(t))=\kappa(t),\qquad t\in(0,T+\delta).\end{align}
The square brackets denote the conformal class. Thus, for example, $[\eta(t)]$ is the set of smooth metrics of the form $\theta\eta(t)$, where $\theta$ is a positive function on $\partial M$. The notation $\mathcal H(\tilde g(t))$ stands for the mean curvature of $\partial M$ with respect to $\tilde g(t)$. By definition, $\mathcal H(\tilde g(t))$ is the trace of $\II(\tilde g(t))$.
We impose the initial condition
\begin{align}\label{tilde_RF_ini}
\tilde g(0)=\hat g.
\end{align}
Here, $\hat g$ is the $G$-invariant metric on $M$ fixed in Section~\ref{subsec_STE_uniq_results}. Equality~\eqref{symm_ans} holds true.
It will be convenient for us to assume that $\tilde g(t)$ is smooth on $M\times[0,T+\delta)$. Remark~\ref{rem_monot_nonsmooth} below explains how this assumption can be relaxed. Note that the paper~\cite{PG12} offers a comprehensive existence theorem for solutions to~\eqref{tilde_RF_eq}--\eqref{tilde_RF_BC}--\eqref{tilde_RF_ini}. Corollary~5.1 in that paper provides a simple sufficient condition for the $G$-invariance of such solutions.

Let us consider the system of equations
\begin{align}\label{MRF_eq}
\frac{\partial}{\partial t}g(t) &= -2(\Ric(g(t)) + \Hess p(t)),\notag\\
\frac{\partial}{\partial t}p(t) &= -\Delta p(t) - R(g(t)).
\end{align}
The unknowns here are the Riemannian metric $g(t)$ and the real-valued function $p(t)$ on $M$ depending on~$t$. The notation $\Hess$ and $\Delta$ refers to the Hessian and the Laplacian with respect to $g(t)$. The relationship between system~\eqref{MRF_eq} and the Ricci flow is well-understood on closed manifolds. It is explained in detail in, e.g.,~\cite[Chapter~6]{PT06}. Essentially, solutions to the first equation of~\eqref{MRF_eq} are pullbacks of solutions to the Ricci flow by $t$-dependent diffeomorphisms. If the pair $(g(t),p(t))$ satisfies system~\eqref{MRF_eq} on a closed manifold, then the expression $\mathcal F(g(t),p(t))$ is non-decreasing in $t$. 

We supplement~\eqref{MRF_eq} with the boundary conditions
\begin{align}\label{MRF_BC}
[g_{\partial M}(t)] = [\eta(t)],\qquad \mathcal H(g(t)) = \kappa(t),\qquad \frac{\partial}{\partial \nu}p(t) = 0.
\end{align}
In the third equality, $\frac{\partial}{\partial \nu}$ denotes differentiation along the outward unit normal vector field on $\partial M$ with respect to~$g(t)$. 
Proposition~\ref{prop_MRF_STE} will explain how the boundary-value problem~\eqref{MRF_eq}--\eqref{MRF_BC} relates to the Ricci flow on~$M$. For an analogous result on closed manifolds, see~\cite[Theorem~6.4.1]{PT06}. In Section~\ref{sec_pfMonF}, we will demonstrate that $\mathcal F$ is monotone on solutions to~\eqref{MRF_eq}--\eqref{MRF_BC}.

\begin{proposition}\label{prop_MRF_STE}
There exist a smooth map $g:M\times[0,T)\to T^*M\otimes T^*M$, a smooth map $\Psi:M\times[0,T)\to M$ and a smooth function $p:M\times[0,T)\to\mathbb{R}$ such that the following statements hold:
\begin{enumerate}

\item
For every $x \in M$ and $t \in (0,T)$, the tensor $g(x,t)$ is a symmetric positive-definite tensor at the point~$x$. Thus, $g(t) = g(\cdot,t)$ is a Riemannian metric on $M$.

\item 
Given $\gamma \in G$, $x \in M$ and $t \in (0,T)$, the equalities $g(t) = \gamma^*g(t)$ and $p(\gamma x, t) = p(x,t)$ hold true.

\item
For some $h,f_1,\ldots,f_n:[0,1]\times(0,T)\to(0,\infty)$, the metric $g(t)$ satisfies~\eqref{symm_met}--\eqref{symm_met_pr}  on $M\times(0,T)$.

\item
The pair $(g(t),p(t))$, where $p(t)$ denotes the function $p(\cdot,t)$, solves system~\eqref{MRF_eq} on $M_0\times(0,T)$.

\item
The boundary conditions~\eqref{MRF_BC} are satisfied on $\partial M\times (0,T)$, and $g(0)=\hat g$ on $M$.

\item
The equality $\Psi(x,t)=x$ holds when $x\in\partial M$ and $t\in(0,T)$. The map $\Psi(\cdot,t)$ is a diffeomorphism of~$M$ for each $t\in(0,T)$. The metric $g(t)$ coincides with the pullback $\Psi^*(\cdot,t)\tilde g(t)$ of the solution $\tilde g(t)$ to problem~\eqref{tilde_RF_eq}--\eqref{tilde_RF_BC}--\eqref{tilde_RF_ini}.
\end{enumerate}
\end{proposition}

\begin{remark}\label{rem_monot_nonsmooth}
We assumed above that $\tilde g(t)$ was smooth on $M\times[0,T+\delta)$. One may obtain results analogous to Proposition~\ref{prop_MRF_STE} under weaker hypotheses. For example, take a natural number $k$ greater than~3. Instead of demanding that $\tilde g$ be smooth on $M\times[0,T+\delta)$, assume it is only smooth on $M_0\times[0,T+\delta)$. In addition, let $\frac{\partial^i}{\partial t^i}\hat\nabla^j\tilde g$ exist and be continuous on $M\times[0,T+\delta)$ whenever $2i+j\le k$. The notation $\hat\nabla$ here stands for the covariant derivative with respect to $\hat g$. Proposition~\ref{prop_MRF_STE} will hold under these hypotheses if one modifies the differentiability properties of $g$, $\Psi$ and $p$ in its formulation. Specifically, one may assert that $g$ is smooth on $M_0\times(0,T)$, while $\frac{\partial^i}{\partial t^i}\hat\nabla^j g$ exists and is continuous on $M\times[0,T)$ when $2i+j\le k-3$. The adjustments required for $\Psi$ and $p$ are of the same nature. We will not discuss them in the present paper.
\end{remark}

\begin{proof}
Lemma~\ref{lem_symm_pres} implies the equality
\begin{align*}
\tilde g(t) = \tilde h^2(r,t)\,dr\otimes dr + \tilde g^r(t),\qquad r\in[0,1],\;t\in[0,T+\delta).
\end{align*}
Here, $\tilde h$ is a positive function, and $\tilde g^r(t)$ is a $t$-dependent $G$-invariant metric on $G/H$. Because Hypothesis~\ref{assum_decomp_p} holds, there exist $\tilde f_1,\dots,\tilde f_n:[0,1]\times[0,T+\delta)\to(0,\infty)$ such that
\begin{align*}
\tilde g^r(t)(X,Y) = \tilde f_1^2(r,t)Q(\pr_{\mathfrak p_1}X, \pr_{\mathfrak p_1}Y)+\dots+\tilde f_n^2(r,t)Q(\pr_{\mathfrak p_n}X, \pr_{\mathfrak p_n}Y),\qquad X,Y \in \mathfrak p.
\end{align*}
In order to construct $g$, $\Psi$ and $p$, we need to introduce auxiliary functions $\tilde p:[0,1]\times[0,T)\to(0,\infty)$ and $\psi:[0,1]\times[0,T)\to[0,1]$. The definitions of $\tilde p$ and $\psi$ will involve $\tilde h,\tilde f_1,\dots,\tilde f_n$.

The scalar curvature of the metric $\tilde g(t)$ at the point $(r,\gamma H) \in M$ does not depend on $\gamma \in G$. Indeed, the Ricci flow equation~\eqref{RF_eq} implies
\begin{align*}
R(\tilde g(t))((r,\gamma H)) = -\frac{\tilde h_t(r,t)}{\tilde h(r,t)}-\sum_{k=1}^nd_k\frac{\tilde f_{kt}(r,t)}{\tilde f_k(r,t)}.
\end{align*}
In what follows, we will abbreviate $R(\tilde g(t))((r,\gamma H)) $ to $\tilde R(r,t)$. We thus obtain a function $\tilde R:[0,1]\times[0,T+\delta)\to\mathbb{R}$. The classical theory of linear parabolic problems (see, e.g.,~\cite[Theorem~5.3 in Chapter~IV and Theorem~12.2 in Chapter~III]{OLVSNU68}) yields the existence of a continuous $\tilde p:[0,1]\times[0,T]\to\mathbb R$, which is smooth on $[0,1]\times[0,T)$, satisfying the equation
\begin{align*}
\tilde p_t(r,t) = - \frac{1}{\tilde h^2(r,t)}\tilde p_{rr}(r,t)+\frac{\tilde h_r(r,t)}{\tilde h^3(r,t)}\tilde p_r(r,t)-\sum_{k=1}^nd_k\frac{\tilde f_{kr}(r,t)}{\tilde h^2(r,t)\tilde f_k(r,t)}\tilde p_r(r,t) &+ \tilde R(r,t)\tilde p(r,t), \\ r\in(0,1),~t&\in(0,T),
\end{align*}
the boundary conditions
\begin{align*}
\tilde p_r(j,t) = 0,\qquad j=0,1,~t\in[0,T),
\end{align*}
and the terminal condition
\begin{align*}
\tilde p(r,T)=1,\qquad r\in[0,1].
\end{align*}
According to the Hopf Lemma, $\tilde p$ must be positive. We define $\psi:[0,1]\times[0,T)\to\mathbb{R}$ by the formula
\begin{align*}
\psi_t(r,t) = \frac{\tilde p_\rho(\rho,t)}{\tilde h^2(\rho,t)\tilde p(\rho,t)}\bigg|_{\rho = \psi(r,t)},\qquad r\in[0,1],\;t\in(0,T),
\end{align*}
and the requirement that $\psi(r,0) = r$ for $r\in[0,1]$. Note that, because $\tilde p_r(j,t)= 0$ when $j=0,1$ and $t\in[0,T)$, the range of $\psi$ is actually the interval $[0,1]$. We will now use the functions $\tilde p$ and $\psi$ to produce $g$, $\Psi$ and $p$. The properties listed in the theorem will be easy to verify.

Let us set
\begin{align}\label{def_mod_h_f}
h(r,t) = \psi_r(r,t)\tilde h(\psi(r,t),t),\qquad f_i(r,t) = \tilde f_i(\psi(r,t),t),\qquad r\in[0,1],\;t\in[0,T),\;i=1,\dots,n.
\end{align}
Consider the $t$-dependent metric $g(t)$ on $M$ given by the formulas~\eqref{symm_met}--\eqref{symm_met_pr}. This metric defines a map $g: M\times[0,T)\to T^*M\otimes T^*M$ in a natural way. Further, we introduce $\Psi:M\times[0,T)\to M$ and $p:M\times[0,T)\to\mathbb R$ through the equalities
\begin{align*}
\Psi((r,\gamma H), t) &= (\psi(r,t), \gamma H), \\
p((r,\gamma H),t) &= -\log\tilde p(\psi(r,t),t),\qquad \gamma \in G,\;r\in[0,1],\;t \in[0,T).
\end{align*}
It is evident that $g$, $\Psi$ and $p$ possess the properties~1, 2, 3 and~6 listed in the proposition. A direct verification demonstrates that the pair $(g(t),p(t))$ satisfies~\eqref{MRF_eq} on $M_0\times(0,T)$. One performs an analogous verification when analysing~\eqref{MRF_eq} on closed manifolds; cf.~\cite[Section~6.4]{PT06}. With the aid of~\eqref{def_mod_h_f}, we find
\begin{align*}
[g_{\partial M}(t)]&=[\tilde g_{\partial M}(t)] = [\eta(t)],\\
\mathcal H(g(t))|_{\{j\}\times G/H} &= \sum_{k=1}^n(-1)^{j+1}d_k\frac{f_{kr}(j,t)}{h(j,t)f_k(j,t)}\\
&=\sum_{k=1}^n(-1)^{j+1}d_k\frac{\tilde f_{kr}(j,t)}{\tilde h(j,t) \tilde f_k(j,t)}\\
&=\mathcal H(\tilde g(t))|_{\{j\}\times G/H} = \kappa(t)|_{\{j\}\times G/H},\qquad j=0,1,\;t\in(0,T).
\end{align*}
Finally,
\begin{align*}
\frac{\partial}{\partial \nu}p(t)|_{\{j\}\times G/H} = \frac{(-1)^j\psi_r(j,t)\tilde p_r(j,t)}{h(j,t)\tilde p(j,t)} = 0,\qquad j=0,1,\;t\in(0,T),
\end{align*}
and $\psi(\cdot,0)$ is the identity map on $[0,1]$. Thus, $g$ and $p$ satisfy statements~4 and~5 in the formulation of the proposition. 
\end{proof}

\subsection{Monotonicity of $\mathcal F$}\label{sec_pfMonF}
The following result demonstrates the connection between~\eqref{MRF_eq}--\eqref{MRF_BC} and the monotonicity of the functional~$\mathcal F$ on $M$.

\begin{theorem}\label{thm_Fmonot}
Consider a smooth map $g:M\times(0,T)\to T^*M\otimes T^*M$ and a smooth function $p:M\times(0,T)\to\mathbb{R}$ satisfying statements~1 through~4 of Proposition~\ref{prop_MRF_STE}. Suppose the boundary conditions~\eqref{MRF_BC} hold on $\partial M\times(0,T)$ with $\eta(t)$ independent of $t$ and $\kappa(t)$ identically equal to~0. Then 
\begin{align*}
\frac{d}{dt}\mathcal F(g(t),p(t)) = 2\int_M |\Ric(g(t))+\Hess p(t)|^2e^{-p(t)}\,d\mu,\qquad t\in(0,T),
\end{align*}
where the absolute value, the Hessian and the volume measure $\mu$ are computed with respect to $g(t)$. Consequently, the quantity $\mathcal F(g(t),p(t))$ is non-decreasing in $t\in(0,T)$.
\end{theorem}

\begin{remark}
While it is convenient for us to assume in Theorem~\ref{thm_Fmonot} that $g$ and $p$ are smooth on $M\times(0,T)$, we can establish the result under weaker hypotheses. It suffices to demand that $\frac{\partial^i}{\partial t^i}\hat\nabla^jg$ and $\frac{\partial^i}{\partial t^i}\hat\nabla^jp$ exist and be continuous on $M\times(0,T)$ when $2i+j\le4$ and $2i+j\le3$, respectively. We remind the reader that $\hat\nabla$ denotes covariant differentiation with respect to $\hat g$.
\end{remark}

\begin{proof}
A computation shows that
\begin{align*}
\frac{d}{dt}\mathcal F(g(t),p(t)) =&\; 2\int_M |\Ric(g(t))+\Hess p(t)|^2e^{-p(t)}\,d\mu\\
&+2\int_{\partial M}(\Div \Ric(g(t)))(\nu)e^{-p(t)}\,d\sigma\\
&-2\int_{\partial M}(\Ric(g(t))+\Hess p(t))(\nu, \nabla p(t))e^{-p(t)}\,d\sigma\\
&+2\int_{\partial M}\frac{\partial}{\partial \nu}R(g(t))e^{-p(t)}\,d\sigma.
\end{align*}
Here, $\Div$ is the divergence with respect to $g(t)$, and $\nu$ is the outward unit normal vector field on $\partial M$ with respect to $g(t)$. The letter $\sigma$ denotes the volume measure of the metric induced by $g(t)$ on $\partial M$. The above formula for $\frac{d}{dt}\mathcal F(g(t),p(t))$ is well-known (see, for example,~\cite{JCAM12} and the related computations in~\cite[Section~6.2]{PT06}). The author first learned of it from Xiaodong Cao's unpublished notes in 2007.

Statement~2 of Proposition~\ref{prop_MRF_STE} implies that, given $t \in (0,T)$, the function $p(t)$ is constant on $\{r\}\times G/H$ for each $r\in[0,1]$. Also, $p(t)$ satisfies the boundary condition $\frac{\partial}{\partial \nu}p(t)=0$. Consequently, the gradient $\nabla p(t)$ vanishes on $\partial M$. This means we can rewrite the above formula for $\frac{d}{dt}\mathcal F(g(t),p(t))$ as
\begin{align*}
\frac{d}{dt}\mathcal F(g(t),p(t)) = 2\int_M|\Ric(g(t))+\Hess p(t)|^2e^{-p(t)}\,d\mu + 2\int_{\partial M}\mathfrak F(g(t))e^{-p(t)}\,d\sigma,
\end{align*}
where
\begin{align}\label{def_frkF}
\mathfrak F(g(t)) = (\Div \Ric(g(t)))(\nu)+\frac{\partial}{\partial \nu}R(g(t)).
\end{align}
To prove the theorem, it suffices to show that $\mathfrak F(g(t)) = 0$ for all $t \in (0,T)$. Note that we could further simplify the expression in the right-hand side of~\eqref{def_frkF} by utilizing the contracted second Bianchi identity. Such a simplification, however, would only hinder our proof. 

The metric $g(t)$ is given by~\eqref{symm_met}--\eqref{symm_met_pr}. Equalities~\eqref{MRF_eq} yield
\begin{align*}
\Ric(g(t)) = \zeta(r,t)\,dr\otimes dr + \Ric^r(g(t)),\qquad r\in[0,1],\;t\in(0,T),
\end{align*}
with $\zeta:[0,1]\times(0,T)\to\mathbb{R}$ a smooth function and $\Ric^r(g(t))$ the $G$-invariant $(0,2)$-tensor field on $G/H$ satisfying
\begin{align*}
\Ric^r(g(t))(X,Y) = &-\sum_{k=1}^nd_k\left(f_k(r,t)f_{kt}(r,t)+\frac{f_k(r,t)f_{kr}(r,t)\bar p_r(r,t)}{h^2(r,t)}\right)Q(\pr_{\mathfrak p_k}X, \pr_{\mathfrak p_k}Y),\qquad X,Y \in \mathfrak p.
\end{align*}
Here, $\bar p:[0,1]\times(0,T)\to\mathbb R$ is such that
\begin{align*}
\bar p(r,t)=p((r,\gamma H),t),\qquad \gamma \in G,\;r\in[0,1],\;t \in(0,T).
\end{align*}
We compute $(\Div \Ric(g(t)))(\nu)$ (cf.~\cite[Lemma 4.2]{AP13}) and find
\begin{align*}
(\Div \Ric(g(t)))(\nu)&|_{\{j\}\times G/H} = (-1)^j\left(\frac{\zeta_r(j,t)}{h^3(j,t)} - \frac{2\zeta(j,t)h_r(j,t)}{h^4(j,t)}\right)\\
&+(-1)^j\sum_{k=1}^nd_k\left(\frac{\zeta(j,t)f_{kr}(j,t)}{h^3(j,t)f_k(j,t)}+\frac{f_{kt}(j,t)f_{kr}(j,t)}{h(j,t)f_k^2(j,t)}\right),\qquad j = 0,1,\;t\in(0,T).
\end{align*}
Also, we have 
\begin{align*}
\frac{\partial}{\partial \nu}R(g(t))|_{\{j\}\times G/H} =&\; \frac{(-1)^{j+1}}{h(j,t)}\left.\left(\frac{\zeta(r,t)}{h^2(r,t)}-\sum_{k=1}^nd_k\left(\frac{f_{kt}(r,t)}{f_k(r,t)}+\frac{f_{kr}(r,t)\bar p_r(r,t)}{h^2(r,t)f_k(r,t)}\right)\right)_r\right|_{r=j} \\
=&\;\frac{(-1)^j}{h(j,t)}\left(\sum_{k=1}^nd_k\frac{f_{kr}(j,t)}{f_k(j,t)}\right)_t + (-1)^j \sum_{k=1}^nd_k \frac{f_{kr}(j,t)\bar p_{rr}(j,t)}{h^3(j,t)f_k(j,t)} \\
&+(-1)^{j+1}\left(\frac{\zeta_r(j,t)}{h^3(j,t)} - \frac{2\zeta(j,t)h_r(j,t)}{h^4(j,t)}\right),\qquad j=0,1,\;t\in(0,T). 
\end{align*}
Therefore, the formula
\begin{align}\label{frkF_bdy}
\mathfrak F(g(t))|_{\{j\}\times G/H} =&\; (-1)^j\sum_{k=1}^nd_k\left(\frac{\zeta(j,t)f_{kr}(j,t)}{h^3(j,t)f_k(j,t)}+\frac{f_{kt}(j,t)f_{kr}(j,t)}{h(j,t)f_k^2(j,t)}\right) \notag\\ &+\frac{(-1)^j}{h(j,r)}\left(\sum_{k=1}^nd_k\frac{f_{kr}(j,t)}{f_k(j,t)}\right)_t + (-1)^j \sum_{k=1}^nd_k \frac{f_{kr}(j,t)\bar p_{rr}(j,t)}{h^3(j,t)f_k(j,t)} \notag\\
=&-\frac{\zeta(j,t)}{h^2(j,t)}\mathcal H(g(t))|_{\{j\}\times G/H} + (-1)^j\sum_{k=1}^nd_k\frac{f_{kt}(j,t)}{f_k(j,t)}\frac{f_{kr}(j,t)}{h(j,t)f_k(j,t)}\notag\\
&-\frac{1}{h(j,t)}(h(j,t)\mathcal H(g(t))|_{\{j\}\times G/H})_t-\frac{\bar p_{rr}(j,t)}{h^2(j,t)}\mathcal H(g(t))|_{\{j\}\times G/H},\qquad j=0,1,\;t\in(0,T),
\end{align}
must hold.

Because $[g_{\partial M}(t)]$ is independent of $t \in (0,T)$, there exist positive functions $\xi_0$ and $\xi_1$ on $(0,T)$ such that
\begin{align*}
f_k(j,t) = \xi_j(t)f_k\biggl(j,\frac T2\biggr),\qquad j=0,1,\;t\in(0,T),\;k=1,\dots,n.
\end{align*}
This implies
\begin{align*}
\frac{f_{kt}(j,t)}{f_k(j,t)} = \frac{\xi_{jt}(t)}{\xi_j(t)},\qquad j=0,1,\;t\in(0,T),\;k=1,\dots,n.
\end{align*}
Consequently,
\begin{align*}
\mathfrak F(g(t))|_{\{j\}\times G/H} =&-\frac{\zeta(j,t)}{h^2(j,t)}\mathcal H(g(t))|_{\{j\}\times G/H} -\frac{\xi_{jt}(t)}{\xi_j(t)}\mathcal H(g(t))|_{\{j\}\times G/H}\notag\\
&-\frac{1}{h(j,t)}(h(j,t)\mathcal H(g(t))|_{\{j\}\times G/H})_t-\frac{\bar p_{rr}(j,t)}{h^2(j,t)}\mathcal H(g(t))|_{\{j\}\times G/H},\qquad j=0,1,\;t\in(0,T).
\end{align*}
The assumption that $\mathcal H(g(t)) = 0$ now yields $\mathfrak F(g(t)) = 0$ for all $t\in(0,T)$. 
\end{proof}

The paper~\cite{JL13} conducts a detailed study of the weighted Gibbons-Hawking-York functional $I_\infty$ on manifolds with boundary. Its Proposition~2 
computes the variation of $I_\infty$. One can derive Theorem~\ref{thm_Fmonot} from that result instead of arguing as above. However, the calculations in the present paper are somewhat simpler than those in~\cite{JL13} because they exploit the symmetries of $M$. Besides, they provide a new formula for the derivative $\frac{d}{dt}\mathcal F(g(t),p(t))$ under the assumptions that $g$ and $p$ are smooth and satisfy statements~1 through~4 of Proposition~\ref{prop_MRF_STE}: equality~\eqref{frkF_bdy} implies
\begin{align*}
\frac{d}{dt}\mathcal F(g(t),p(t)) =&\; 2\int_M|\Ric(g(t))+\Hess p(t)|^2e^{-p(t)}\,d\mu\\
&+2\int_{\partial M}(- \Ric(g(t))(\nu,\nu)\mathcal H(g(t))+\langle (\Ric(g(t)))_{\partial M},\II(g(t))\rangle )\,e^{-p(t)}\,d\sigma\\
&-2\int_{\partial M}\bigg(\frac{1}{|\hat \nu|}(|\hat \nu|\mathcal H(g(t)))_t
+\Delta p(t)\mathcal H(g(t))\bigg) d\sigma 
\\
=&\; 2\int_M|\Ric(g(t))+\Hess p(t)|^2\,e^{-p(t)}\,d\mu\\
&+2\int_{\partial M}(\langle (\Ric(g(t)))_{\partial M},\II(g(t))\rangle-(\mathcal H(g(t)))_t)\,e^{-p(t)}\,d\sigma,\qquad t\in(0,T).
\end{align*}
The angular brackets here mean the scalar product in the tensor bundle over $\partial M$ induced by $g_{\partial M}(t)$. Interpreting $\Ric(g(t))$ as a map from $TM\otimes TM$ to $\mathbb R$, we write $(\Ric(g(t)))_{\partial M}$ for the restriction of this map to $T\partial M\otimes T\partial M$. The notation $\hat \nu$ stands for the outward unit normal vector field on $\partial M$ with respect to the metric $\hat g$.

\section*{Acknowledgements}

I am grateful to Xiaodong Cao and Jeff Streets for the stimulating discussions on monotone functionals and the Ricci flow.


\begin{thebibliography}{99}

\bibitem{HA90}
H. Amann, Dynamic theory of quasilinear parabolic equations — II. Reaction-diffusion systems, Diff. Int. Equ.~3~(1990) 13--75. 

\bibitem{MA08}
M.T. Anderson, On boundary value problems for Einstein metrics,
Geom. Topol.~12 (2008) 2009--2045.

\bibitem{MBXCAP10}
M.~Bailesteanu, X.~Cao, A.~Pulemotov, Gradient estimates for the
heat equation under the Ricci flow, J.~Funct. Anal.~258 (2010)
3517--3542.

\bibitem{MB14}
M. Buzano, Ricci flow on homogeneous spaces with two isotropy summands,
Ann. Global Anal. Geom.~45 (2014) 25--45. 

\bibitem{XCTD06}
X.-Z. Chen, T. Dong, Ricci deformation of a metric on a
Riemannian manifold with boundary, J.~Zhejiang Univ. Sci.
Ed.~33 (2006) 496--499.

\bibitem{BCPLLN06}
B. Chow, P. Lu, L. Ni, Hamilton's Ricci flow, Amer. Math.
Soc., Providence,~RI, 2006.

\bibitem{JC07}
J.C. Cortissoz, The Ricci flow on the two ball with a
rotationally symmetric metric, Russian Math. (Iz. VUZ)~51
(2007), no. 12, 30--51.

\bibitem{JC09}
J.C. Cortissoz, Three-manifolds of positive curvature and convex
weakly umbilic boundary, Geom. Dedicata~138 (2009) 83--98.

\bibitem{JCAM12}
J. Cortissoz, A. Murcia, The Ricci flow on surfaces with boundary,
arXiv:1209.2386 [math.DG].

\bibitem{ADMW99}
A.S. Dancer, M.Y. Wang, Integrable cases of the Einstein equations,
Comm. Math. Phys.~208 (1999) 225--243.

\bibitem{ADMW00}
A.S. Dancer, M.Y. Wang, The cohomogeneity one Einstein equations
from the Hamiltonian viewpoint, J.~reine angew. Math.~524 (2000)
97--128.

\bibitem{ADMW11}
A.S. Dancer, M.Y. Wang, On Ricci solitons of cohomogeneity one, Ann.
Global Anal. Geom.~39 (2011) 259--292.

\bibitem{PG12}
P. Gianniotis, The Ricci flow on manifolds with boundary,
arXiv:1210.0813 [math.DG].

\bibitem{PG13}
P. Gianniotis, Boundary estimates for the Ricci flow, arXiv:1311.3162 [math.DG].

\bibitem{KGWZ02}
K. Grove, W. Ziller, Cohomogeneity one manifolds with positive Ricci
curvature, Invent. Math.~149 (2002) 619--646.

\bibitem{CH10}
C.A. Hoelscher, Classification of cohomogeneity one manifolds in low
dimensions, Pacific J.~Math.~246 (2010) 129--185.

\bibitem{OLVSNU68}
O.A. Lady\v{z}enskaja, V.A. Solonnikov,  N.N. Ural'ceva, Linear and quasilinear equations of parabolic type, Amer. Math. Soc., Providence, RI, 1968.

\bibitem{JoLa13}
J. Lauret, Ricci flow of homogeneous manifolds, Math.~Z.~274 (2013) 373--403. 

\bibitem{JL13}
J. Lott, Mean curvature flow in a Ricci flow background,
Comm. Math. Phys.~313 (2012) 517--533.

\bibitem{JLNS14}
J. Lott, N. Sesum, Ricci flow on three-dimensional manifolds with symmetry, Comment. Math. Helv.~89 (2014) 1--32.

\bibitem{JMGT07}
J. Morgan, G. Tian, Ricci flow and the Poincar\'{e}
Conjecture, Amer. Math. Soc., Providence,~RI; Clay Mathematics
Institute, Cambridge,~MA, 2007.

\bibitem{AP10}
A. Pulemotov, Quasilinear parabolic equations and the Ricci flow on
manifolds with boundary, J.~reine angew. Math.~683 (2013) 97--118.

\bibitem{AP11}
A. Pulemotov, Metrics with prescribed Ricci curvature near the
boundary of a manifold, Math. Ann.~357 (2013) 969--986.

\bibitem{AP13}
A. Pulemotov, The Dirichlet problem for the prescribed Ricci curvature equation on cohomogeneity one manifolds, submitted, arXiv:1303.2419 [math.AP].

\bibitem{YS92}
Y. Shen, New results on some dynamical and stationary
problems in geometry, Ph.D.~Dissertation, Stanford University, 1992.

\bibitem{YS96}
Y. Shen, On Ricci deformation of a Riemannian metric on manifold
with boundary, Pacific J.~Math.~173 (1996) 203--221.

\bibitem{PT06}
P. Topping, Lectures on the Ricci flow, Cambridge
University Press, Cambridge, 2006.

\bibitem{PW91}
P. Weidemaier, Local existence for parabolic problems with
fully nonlinear boundary condition; an $L_p$-approach,  Ann. Mat.
Pura Appl.~160 (1991) 207--222.

\end{thebibliography}
\end{document}